\documentclass[12pt]{amsart}
\usepackage{verbatim}
\usepackage{amssymb}
\usepackage{enumerate}
\usepackage[active]{srcltx}
\numberwithin{equation}{section}

\usepackage{t1enc}

\usepackage[utf8x]{inputenc}
\usepackage{color}
\newtheorem{theorem}{Theorem}
\newtheorem{proposition}[theorem]{Proposition}
\newtheorem{lemma}[theorem]{Lemma}

\newtheorem{problem}{Problem}

\newtheorem{example}[theorem]{Example}

\theoremstyle{definition}

\newtheorem{definition}[theorem]{Definition}

\newcommand{\setm}{\setminus}
\newcommand{\empt}{\emptyset}
\newcommand{\subs}{\subset}

\def\<{\left\langle}
\def\>{\right\rangle}
\def\br#1;#2;{\bigl[ {#1} \bigr]^ {#2} }

\begin{document}

\author[I. Juh\'asz]{Istv\'an Juh\'asz}
\address
      { Alfréd Rényi Institute of Mathematics, Hungarian Academy of Sciences}
\email{juhasz@renyi.hu}

\author[L. Soukup]{Lajos Soukup}
\address
      { Alfréd Rényi Institute of Mathematics, Hungarian Academy of Sciences}
\email{soukup@renyi.hu}

\author[Z. Szentmikl\'ossy]{Zolt\'an Szentmikl\'ossy}
\address{E\"otv\"os University of Budapest}
\email{szentmiklossyz@gmail.com
}

\title[First countable and almost discretely Lindelöf spaces]
{First countable and almost discretely Lindelöf $T_3$ spaces
have cardinality\\ at most continuum}

\date{\today}
\thanks{The research on and preparation of this paper was
supported by  NKFIH grant no. K 113047.}

\maketitle

\begin{abstract}
A topological space $X$ is called {\em almost discretely Lindelöf} if every discrete set $D \subs X$
is included in a Lindelöf subspace of $X$. We say that the space $X$ is {\em $\mu$-sequential}
if for every non-closed set $A \subs X$ there is a sequence 
of length $\le \mu$ in $A$ that converges to a point which is not in $A$. 
With the help of a technical theorem that involves elementary submodels, we establish the following
two results concerning such spaces.

\begin{enumerate}
\item For every almost discretely Lindelöf $T_3$ space $X$ we have $|X| \le 2^{\chi(X)}$.

\item If $X$ is a $\mu$-sequential $T_2$ space of pseudocharacter
$\psi(X) \le 2^\mu$ and for every free set $D \subs X$ we have
$L(\overline{D}) \le \mu$, then $|X| \le 2^\mu$.
\end{enumerate}

The case $\chi(X) = \omega$ of (1) provides a solution to Problem 4.5 of \cite{JTW},
while the case  $\mu = \omega$ of (2) is a partial improvement on the main result of
\cite{AB}.
\end{abstract}

\bigskip

Our main aim in this note is to prove what is stated in the title and thus
give a solution to problem 4.5 of \cite{JTW}.

All spaces in here are assumed to be $T_1$. Consequently, if $X$ is any space
and $A$ is any subset of $X$ then the pseudocharacter $\psi(A,X)$ of
$A$ in $X$, i.e. the smallest size of a family of open sets whose intersection
is $A$, is well-defined.

We recall that a transfinite sequence $\{x_\alpha : \alpha < \eta\} \subs X$
is called a free sequence in $X$ if for every $ \beta < \eta$ we have

$$\overline{\{x_\alpha : \alpha < \beta\}} \cap \overline{\{x_\alpha : \beta \le \alpha < \eta\}} = \emptyset .$$

We say that a subset $D \subs X$ is free if it has a well-ordering that turns it into
a free sequence. Clearly, every free set is discrete and every countable discrete set
is free in $X$. We shall use $\mathcal{F}(X)$ to denote the family of all free subsets of $X$.
The freeness number $F(X)$ of $X$ is  defined by $F(X) = \sup\{|D| : D \in  \mathcal{F}(X)\}$, while
its hat version $\widehat{F}(X)$ is the smallest cardinal $\kappa$ such that $X$ has no free
subset of size $\kappa$.

Given a space $X$ and an infinite cardinal $\kappa$, we are going to consider elementary submodels
$M$ of $H(\lambda)$ for a large enough regular cardinal $\lambda$ such that $X \in M$ and $M$ is
$<\kappa$-closed, i.e. $M^\varrho \subs M$ for all $\varrho < \kappa$. (Note that by $X \in M$
we really mean $\<X,\,\tau\> \in M$ where $\tau$ is the topology on $X$.) We may also assume,
without any loss of generality, that $\mu + 1 \subs M$ where $\mu = |M|$. The following proposition
is an easy consequence of standard cardinal arithmetic.

\begin{proposition}\label{pr:k}
The minimum cardinality of a $\,<\kappa$-closed elementary submodel $M$ (of some $H(\lambda)$) is
$2^{<\kappa}$ if $\kappa$ is regular and $2^\kappa$ if $\kappa$ is singular.
\end{proposition}

We now present a (somewhat technical) result that, in addition to a space $X$ and a cardinal $\kappa$,
involves such an elementary submodel. This result plays a crucial role in the proof of our main result
and we suspect that it will have numerous other interesting consequences as well.

\begin{theorem}\label{tm:model}
Fix a space $X$ and a cardinal $\kappa$, moreover let $M$ be a $<\kappa$-closed elementary submodel
(of some $H(\lambda)$) such that $X \in M$ and $\mu + 1 \subs M$ where $\mu = |M|$.
If for every $D \in \mathcal{F}(X) \cap [X \cap M]^{<\kappa}$ we have $\psi(\overline{D},X) \le \mu$,
then either $$X = \bigcup\{\overline{D} : D \in \mathcal{F}(X) \cap [X \cap M]^{<\kappa}\}$$
or $\,\widehat{F}(X) > \kappa$, i.e. there is a free set of size $\kappa$ in $X$.
\end{theorem}

\begin{proof}
Let us put $Y = \bigcup\{\overline{D} : D \in \mathcal{F}(X) \cap [X \cap M]^{<\kappa}\}$ and assume that
$X \ne Y$. We may then fix a point $p \in X \setm Y$. Note that, as $M$ is $<\kappa$-closed, we have
$\mathcal{F}(X) \cap [X \cap M]^{<\kappa} \subs M$, hence for every $D \in \mathcal{F}(X) \cap [X \cap M]^{<\kappa}$
we have $\overline{D} \in M$ as well. Consequently, by elementarity, for every such $D$ there is a family
of open sets $\mathcal{V}_D \in M$ such that $\bigcap \mathcal{V}_D = \overline{D}$ and $|\mathcal{V}_D| \le \mu$.
Note that then we have  $\mathcal{V}_D \subs M$ as well.

We are now going to define points $x_\alpha \in X \cap M$ and open sets $V_\alpha \in M$ by transfinite
recursion on $\alpha < \kappa$  so that the following three conditions hold true for every $\alpha < \kappa$:

\begin{enumerate}[(i)]
\item $p \notin V_\alpha$;

\item $\overline{\{x_\beta : \beta < \alpha\}} \subs V_\alpha$;

\smallskip

\item $\{x_\beta : \alpha \le \beta < \kappa\} \cap V_\alpha = \empt$.
\end{enumerate}
Clearly then $\{x_\alpha : \alpha < \kappa\}$ will be a free sequence of length $\kappa$.

So, assume that $\alpha < \kappa$ and we have defined $D_\alpha = \{x_\beta : \beta < \alpha\} \subs X \cap M$
and $\{V_\beta : \beta < \alpha\} \subs \tau \cap M$ such that for every $\gamma < \alpha$ we have $p \notin V_\gamma$,
$\,\overline{D_\gamma} \subs V_\gamma$, and $\{x_\beta : \gamma \le \beta < \alpha\} \cap V_\gamma = \empt$.

Then we have $D_\alpha \in \mathcal{F}(X) \cap [X \cap M]^{<\kappa}$, and so there is $V_\alpha \in \mathcal{V}_{D_\alpha}$
such that $p \notin V_\alpha$. As $M$ is $<\kappa$-closed, we then have $\{V_\beta : \beta \le \alpha\} \in M$ as well and
so $p \notin \bigcup\{V_\beta : \beta \le \alpha\}$ implies $X \cap M \setm \bigcup\{V_\beta : \beta \le \alpha\} \ne \emptyset$
by elementarity. We then pick $x_\alpha$ as any element of $X \cap M \setm \bigcup\{V_\beta : \beta \le \alpha\}$.
It is clear that this recursive procedure carries though for all $\alpha < \kappa$, moreover
$\{x_\alpha : \alpha < \kappa\}$ and $\{V_\alpha : \alpha < \kappa\}$ satisfy the three conditions (i) -- (iii).
\end{proof}

We now turn to applying Theorem \ref{tm:model} to produce some new cardinal function inequalities,
in particular the statement formulated in the title. Our notation and terminology of cardinal functions
follow those in \cite{J}.

We recall, see \cite{JTW}, that a space $X$ is called (almost) discretely Lindelöf if for every discrete set $D \subs X$
we have that $\overline{D}$ is Lindelöf (resp. $D$ is included in a Lindelöf subspace of $X$).
We now present several lemmas concerning such spaces. We shall use $\mathcal{D}(X)$ to denote the family of all discrete
subspaces of $X$. The spread $s(X)$ of $X$ is then defined by $s(X) = \sup\{|D| : D \in \mathcal{D}(X)\}$.

\begin{lemma}\label{lm:s}
For every almost discretely Lindelöf $T_2$ space $X$ we have (i) $s(X) \le 2^{t(X)\cdot\psi(X)}$
and (ii) $F(X) \le t(X)$.
\end{lemma}

\begin{proof}
For every $D \in \mathcal{D}(X)$ there is a Lindelöf $Y \subs X$ with $D \subs Y$. Consequently, we have
$|D| \le |Y| \le 2^{t(X)\cdot\psi(X)}$ by Shapirovskii's theorem from \cite{Sh}, see also 2.27 of \cite{J}.
This proves (i).

To see (ii), assume, arguing indirectly, that $D = \{x_\alpha : \alpha < t(X)^+\}$ is a free sequence in $X$.
Then again there is a Lindelöf $Y \subs X$ with $D \subs Y$ and so $D$ has a complete accumulation
point $y$ in $Y$. But then, by the definition of $t(X)$, there is an $\alpha < t(X)^+$ such that
$y \in \overline{\{x_\beta : \beta < \alpha\}}$ and at the same time
$y \in \overline{\{x_\beta : \alpha \le \beta < t(X)^+\}}$, contradicting that $\{x_\alpha : \alpha < t(X)^+\}$
is a free sequence. This proves (ii).
\end{proof}

The cardinal function $g(X) = \sup\{|\overline{D}| : D \in \mathcal{D}(X)\}$ was introduced in \cite{A}
(and is not mentioned in \cite{J}). Since every right separated (or equivalently: scattered) space
has a dense discrete subspace, for every space $X$ we have $h(X) \le g(X)$ because $h(X)$ is just the
supremum of the sizes of all right separated subspaces of $X$.

We are now ready to formulate and prove our main result.

\begin{theorem}\label{tm:chi}
For every almost discretely Lindelöf $T_3$ space $X$ we have $|X| \le 2^{\chi(X)}$.
\end{theorem}

\begin{proof}
We first show that $g(X) \le 2^{\chi(X)}$. Indeed, if $D \in \mathcal{D}(X)$ then
we have $|D| \le 2^{t(X)\cdot\psi(X)} \le 2^{\chi(X)}$ by part (i) of Lemma \ref{lm:s}.
But we also have $|\overline{D}| \le |D|^{\chi(X)}$, see e.g. 2.5 of \cite{J},
hence putting these together we get $|\overline{D}| \le 2^{\chi(X)}$ as well.
This, in turn, yields us $h(X) \le g(X) \le 2^{\chi(X)}$.

But as $X$ is $T_3$, we have $\Psi(X) \le h(X) = hL(X)$, i.e. for every closed
set $H \subs X$ we have $\psi(H,X) \le h(X)$. Indeed, every point $x \in X \setm H$
admits a neighborhood $U_x$ such that $\overline{U_x} \cap H = \empt$ and then there 
is a set $A \subs X \setm H$ with $|A| \le h(X)$ such that 
$X \setm H =  \bigcup\{U_x : x \in A\} = \bigcup\{\overline{U_x} : x \in A\}$.
This is the point where the $T_3$
property, and not just $T_2$, is essentially used.

And now we apply Theorem \ref{tm:model} to $X$ with the choice $\kappa = \chi(X)^+$,
by choosing an appropriate elementary submodel $M$ of cardinality $2^{\chi(X)}$
that is $\chi(X)$-closed (i.e. ${<\chi(X)^+}$-closed), moreover $X \in M$ and
$2^{\chi(X)}+1 \subs M$. Note that we have established above the inequality
$\Psi(X) \le 2^{\chi(X)}$ that is much more than what is needed to
ensure the applicability of Theorem \ref{tm:model}.

But by part (ii) of Lemma \ref{lm:s} we have $F(X) \le t(X) \le \chi(X)$, i.e.
there is no free set in $X$ of size $\kappa = \chi(X)^+$, consequently
$$X = \bigcup\{\overline{D} : D \in \mathcal{F}(X) \cap M\}$$
must be satisfied.

Finally, using again $g(X) \le 2^{\chi(X)}$ we can conclude from the above equality
and from $|M| = 2^{\chi(X)}$ that $|X| \le 2^{\chi(X)}$.
\end{proof}

We do not know if the $T_3$ property can be relaxed to the $T_2$ property
in Theorem \ref{tm:chi}, even in the countable case $\chi(X) = \omega$.
However, it should be mentioned concerning this that Spadaro has given
a consistent affirmative answer to this question in \cite{S}. He proved
in fact that if $\mathfrak{c} = 2^{<\mathfrak{c}}$ then every sequential $T_2$
space of pseudocharacter $\le \mathfrak{c}$ which is almost discretely Lindelöf
has cardinality $\le \mathfrak{c}$.

Our next application of Theorem \ref{tm:model} also concerns sequential $T_2$ spaces,
in fact a generalization of sequentiality will be used. Instead of the almost discretely Lindelöf
property, however, a slightly weakened version of the discrete Lindelöf property 
will be used. On the other hand, we shall obtain a ZFC result.

\begin{definition}
Let $\mu$ be an infinite cardinal number. A space $X$ is called {\em $\mu$-sequential}
if for every non-closed set $A \subs X$ there is a (transfinite) sequence
of length $\le \mu$ of points of $A$
that converges to a point which is not in $A$. Thus, $\omega$-sequential $\equiv$ sequential.
\end{definition}

The following proposition is well-known in the case $\mu = \omega$ and hence its proof,
being a natural adaptation of the countable case, is left to the reader.

\begin{proposition}\label{pr:seq}
If $X$ is a $\mu$-sequential $T_2$ space then (i) $t(X) \le \mu$ and (ii) for every set $A \subs X$
we have $|\overline{A}| \le |A|^\mu$.
\end{proposition}

\begin{theorem}\label{tm:psi}
Let $X$ be a $\mu$-sequential $T_2$ space such that for every $D \in \mathcal{F}(X)$ we have
$L(\overline{D}) \le \mu$, moreover $\psi(X) \le 2^\mu$. Then actually $|X| \le 2^\mu$.
\end{theorem}

\begin{proof}
Let us start by noting that, using $t(X) \le \mu$, the same argument as in the proof of part (ii) of Lemma
\ref{lm:s} yields $F(X) \le \mu$. Thus, by part (ii) of Proposition \ref{pr:seq},  for
every free set $D \in \mathcal{F}(X)$ we have $|\overline{D}| \le \mu^\mu = 2^\mu$.

We claim that then  $\psi(\overline{D},X) \le 2^\mu$ also holds for
every free set $D \in \mathcal{F}(X)$. To see this, let us fix for every point $x \in \overline{D}$
a family $\mathcal{U}_x$ of open neighborhoods of $x$ with $|\mathcal{U}_x| \le 2^\mu$
and  $\bigcap\mathcal{U}_x=\{x\}$ and then put
$\mathcal{U} = \bigcup \{\mathcal{U}_x : x \in \overline{D}\}$. Clearly, we have $|\mathcal{U}| \le 2^\mu$
as well. Now, if we fix any point
$p \in X \setm \overline{D}$ then for every $x \in \overline{D}$ there is $U_x \in \mathcal{U}_x$
such that $p \notin U_x$. But then $L(\overline{D}) \le \mu$ implies that for some $A \subs \overline{D}$
with $|A| \le \mu$ we have $\overline{D} \subs \bigcup \{U_x : x \in A\}$. This shows that the family
$$\mathcal{W} = \{\cup \mathcal{V} : \mathcal{V} \in [\mathcal{U}]^{\le \mu} \mbox{ and } \overline{D} \subs \cup \mathcal{V}\}$$
of open sets satisfies $\cap \mathcal{W} = \overline{D}$ and, as clearly $|\mathcal{W}| \le 2^\mu$, the family $\mathcal{W}$
witnesses $\psi(\overline{D},X) \le 2^\mu$.

Consequently, if $M$ is an appropriate $\mu$-closed elementary submodel of cardinality $2^\mu$
such that $\{X\} \cup 2^\mu + 1 \subs M$ then the assumptions of Theorem \ref{tm:model} are
satisfied with $\kappa = \mu^+$. Thus, from $F(X) \le \mu$ we conclude that
$$X = \bigcup\{\overline{D} : D \in \mathcal{F}(X) \cap M\}.$$
But then $|M| = 2^\mu$ and $|\overline{D}| \le 2^\mu$ for all $D \in \mathcal{F}(X)$
clearly imply $|X| \le 2^\mu$.
\end{proof}

Arhangel'skii and Buzyakova has shown in \cite{AB} that the
cardinality of a sequential linearly Lindelöf Tikhonov space
$X$ does not exceed $2^\omega$ if the pseudocharacter of
$X$ does not exceed $2^\omega$. Clearly, the
case  $\mu = \omega$ of Theorem \ref{tm:psi} is a partial improvement on their result.

It is obvious that radial spaces of tightness $\le \mu$ are $\mu$-sequential but it is
also clear that the converse of this statement fails, as is demonstrated by the 
existence of sequential spaces that are not Fr\`echet.

On the other hand, every $\mu$-sequential space is pseudoradial and has tightness $\le \mu$.
Our next example shows that, at least consistently, the converse of this statement also fails,
even for compact spaces.

\begin{example}
The one-point compactification $X = K \cup \{p\}$ of the so called Kunen line $K$, 
constructed from CH in \cite{JKR}, is
pseudoradial and hereditarily separable, hence countably tight, but not sequential.
\end{example}

\begin{proof}
$X$ is hereditarily separable because $K$ is. $K$ is a non-closed subset of $X$ that is sequentially
closed in $X$ because $K$ is countably compact, hence $X$ is not sequential.

Finally, assume that $A \subs X$ is not closed in $X$. If there is a point $x \in K$ with
$x \in \overline{A} \setm A$ then an $\omega$-sequence in $A$ converges to $x$ because $K$
is first countable. Otherwise, $A \subs K$ is closed in $K$ and $\overline{A} = A \cup \{p\}$, hence
$A$ is not compact. But every countable closed subset of $K$ is compact, again by the countable 
compactness of $K$, hence we have $|A| = \omega_1$. Every neighborhood of $p$ in $X$ is
clearly  co-countable, hence we have $\psi(p,A \cup \{p\}) = \chi(p,A \cup \{p\}) = \omega_1$,
and so there is an $\omega_1$-sequence in $A$ that converges to $p$. 
\end{proof}

The following intriguing problem, however remains open.

\begin{problem}
Is there a ZFC example of a countably tight pseudoradial space that is not sequential?
\end{problem}

\end{document}